\documentclass[a4paper]{amsart}
\usepackage{times}
\usepackage{cite}
\usepackage{amsthm,amsmath,amssymb,amsfonts}
\usepackage{algorithmic}
\usepackage{graphicx}
\usepackage{textcomp}
\usepackage{mfl}
\usepackage{times}
\usepackage{cite}
\usepackage{amsthm,amsmath,amssymb,amsfonts}
\usepackage{algorithmic}
\usepackage{graphicx}
\usepackage{textcomp}
\usepackage{mfl}
\usepackage{mathtools}
\usepackage{mathtools}
\usepackage{amsmath}
\usepackage{amssymb, hyperref}
\usepackage{tikz}
\usepackage{color, dsfont}
\usepackage{graphicx}

\newcommand{\Mod}{\ensuremath{\mathrm{Mod}}}
\usetikzlibrary{arrows,chains,matrix,positioning,scopes}
\def\BibTeX{{\rm B\kern-.05em{\sc i\kern-.025em b}\kern-.08em
    T\kern-.1667em\lower.7ex\hbox{E}\kern-.125emX}}
\makeatletter
\tikzset{join/.code=\tikzset{after node path={\ifx\tikzchainprevious\pgfutil@empty\else(\tikzchainprevious)edge[every join]#1(\tikzchaincurrent)\fi}}}
\makeatother
\tikzset{>=stealth',every on chain/.append style={join},
         every join/.style={->}}
\tikzstyle{labeled}=[execute at begin node=$\scriptstyle,
   execute at end node=$]

\newtheorem{Def}{Definition}
\newtheorem{Thm}{Theorem}
\newtheorem{Rmk}[Thm]{Remark}
\newtheorem{Exm}[Thm]{Example}
\newtheorem{Cor}[Thm]{Corollary}
\newtheorem{Lem}[Thm]{Lemma}

\theoremstyle{plain}

\theoremstyle{definition}
\begin{document}

\title{Asymptotic truth-value laws in many-valued logics}

\author{Guillermo Badia}
\address{
School of Historical and Philosophical Inquiry\\
 University of Queensland\\ 
 St Lucia, QLD 4072, Brisbane, Australia\\ 
\texttt{guillebadia89@gmail.com} }
\author{Xavier Caicedo}
\address{Departamento de Matem\'aticas\\ Universidad de los Andes \\ Carrera 1 N. 18 A -70\\ Bogot\'a, Colombia\\ \texttt{xcaicedo@uniandes.edu.co}}
\author{Carles Noguera }
\address{Department of Information Engineering and Mathematics\\University of Siena,
          Via Roma~56, 53100~Siena, Italy\\
           \texttt{carles.noguera@unisi.it}}\date{}
\maketitle

\begin{abstract}
This paper studies which truth-values are most likely to be taken on finite models by arbitrary sentences of a many-valued predicate logic. We obtain generalizations of Fagin's classical zero-one law for any logic with values in a finite lattice-ordered algebra, and for some infinitely valued logics, including \L ukasiewicz logic. The finitely valued case is reduced to the classical one through a uniform translation and Oberschelp's generalization of Fagin's result. Moreover, it is shown that the complexity of determining the almost sure value of a given sentence is PSPACE-complete, and for some logics we may describe completely the set of truth-values that can be taken by sentences almost surely. 

\medskip

\end{abstract}

%\tableofcontents

\section{Introduction}

Given a property $P$ of finite structures, a natural question has been often considered (beginning with
~\cite{erdos} in the context of graph theory): what is the probability that a structure satisfies $P$ when randomly selected among finite structures with the same domain for a suitable probability measure? More interestingly, what are the asymptotic probabilities? Or, more precisely, what do these probabilities converge to (if anything) as the size of the domain of the structures grows to infinite? Sometimes, when the properties under consideration are expressible by formulas of a given logic, we may have a `zero-one law':  the probabilities converge to either $0$ or $1$ (and so we can say that the formula is either \emph{almost surely false} or \emph{almost surely true}). After early preliminary developments for monadic predicate logic~\cite{carnap}, the topic of logical zero-one laws was properly started independently in the papers by Glebskiĭ et al~\cite{gle} and Fagin~\cite{fag} where they establish a zero-one law for first-order classical logic on finite purely relational vocabularies. The argument in~\cite{gle} proceeds by a rather involved quantifier elimination, whereas the one in~\cite{fag} makes use of the so called `extension axioms' that form a complete theory axiomatizing a countable random structure (Lynch in~\cite{lynch} attributes the discovery of this theory to Stanis\l aw Ja\'skowski; both the theory and the random structure appeared first in print in~\cite{gai}, where Gaifman acknowledged that the latter was an example suggested to him by Michael Rabin and Dana Scott). An important related result was obtained by Etienne Grandjean in~\cite{grand} when he proved that the set of almost surely true formulas is PSPACE-complete. The reader may consult~\cite{gu} for a didactic introduction to zero-one laws and~\cite{com} for a rather comprehensive survey. 

Zero-one laws have been only sporadically investigated in the the setting of many-valued logics,  although they offer a number of intriguing issues that are invisible in the  two-valued context. For example, do we have general asymptotic truth-value laws for any many-valued logic? That is, in the presence of a multiplicity of truth-values, are some of the values taken by sentences almost surely? And exactly which of them? And, whenever we do have such a law, what is the complexity of determining the truth-value that an arbitrary formula almost surely takes?  Zero-one laws have indeed been obtained for some finitely valued fuzzy logics in~\cite{bano, kos}, but, to the best of our knowledge, zero-one laws for infinitely valued logics remain largely unexplored. In the related context of semiring semantics, analogous issues have been recently put forward and successfully addressed by Gr\"adel et al in~\cite{gra} for sentences given in negation normal form (involving only the additional connectives of disjunction and conjunction) and letting interpretations assign truth-values directly to the literals (both atomic and negated atomic formulas) of the logic.

In this paper, we consider asymptotic truth-value laws in many-valued logics with a general approach that allows for arbitrary probability measures and arbitrary lattice-based algebras, obtaining as main results:
\begin{itemize}
\item a general zero-one law for finitely valued predicate logics, and
\item a zero-law for infinitely valued \L ukasiewicz logic and other $[0,1]$-valued predicate logics.
\end{itemize}
Moreover, for finitely values logics, we provide examples in which all values are almost sure and examples in which only a few are, and we generalize Grandjean's theorem on the complexity of asymptotic truth-values.

The paper is organized as follows. We recall in \S \ref{s:0-1first} some necessary notions from classical finite model theory and its zero-one laws. In \S \ref{s:finitely} we present our general approach to these issues in the case of {\em finitely} valued logics: in \S \ref{ss:01law-finitely} we obtain the general zero-one law for finitely valued logics, in \S \ref{ss:finitely-random} we discuss random models (which are behind the proof of the previous result) and their axiomatizations, and in \S \ref{ss:complexity} we show that the complexity of the problem of determining the asymptotic truth-value of an arbitrary formula and is in PSPACE and, actually, it is PSPACE-complete in the presence of crisp identity, hence generalizing Grandjean's result for classical predicate logic. In \S \ref{s:SetAlmostSure} we provide natural examples of logics in which {\em all} truth-values are actually almost sure, and a rather wide class of logics in which the almost sure values are few. In \S \ref{s:01-law-Luk} we prove the zero-one law for {\em infinitely} valued \L ukasiewicz predicate logic. In \S \ref{s:Further-infinitely} we consider again the problem of describing sets of almost sure values, now for {\em infinitely} valued logics, and we prove a zero-one law for a large family of $[0,1]$-valued logics satisfying De Morgan laws. In particular, we obtain that all rational numbers in $[0,1]$ are almost sure values of infinitely valued \L ukasiewicz predicate logic. We end the paper in \S \ref{s:Conclusion} with some concluding remarks and open questions.

For readers less familiar with the literature of many-valued logic, a modern comprehensive reference in the general field of mathematical fuzzy logic is~\cite{Cintula-FHN:HBMFL}, while \cite{haj} remains even today a rather useful source that contains detailed proofs of many essential results; we recommend~\cite{mun2} as a nice article highlighting the importance of infinitely valued \L ukasiewicz logic among other non-classical logics.

\section{Zero-one laws in two-valued predicate logic}\label{s:0-1first}
In this section, we recall some facts from classical finite model theory that we aim to generalize later to the context of many-valued logics.

Consider first a purely relational vocabulary $\tau$ for the usual finitary first-order logic that we will denote by $\mathcal{L}_{\omega\omega}$, as it is customary in abstract model theory~\cite{barfer} and even in some works on many-valued logics (see e.g.~\cite{bano2}), where the $\omega$ subscripts represent the finitary character of, respectively, conjunctions/disjunctions and quantifier strings. Sometimes we will also refer to the stronger logics $\mathcal{L}_{\omega_1\omega}$ and $\mathcal{L}_{\infty\omega}$ which allow for conjunctions and disjunctions of, respectively, countably many or arbitarily many formulas, while keeping the finitary character of quantifier strings.

A $\tau$-sentence is said to be \emph{parametric} in the sense of~\cite[p.\ 277]{ober} (or, alternatively, it is an \emph{Oberschelp condition}) if it is equivalent to a finite conjunction of first-order formulas of the form: 
\begin{equation*}
\forall^{\not=}x_{1},\dots ,x_{k}\, \phi (x_{1},\dots ,x_{k}),
\end{equation*}
which abbreviates $\forall x_{1}\dots \forall x_{k}\,(\bigwedge\nolimits_{i<j}\lnot (x_{i}=x_{j})\rightarrow \phi(x_{1},\dots ,x_{k}))$, where $\phi (x_{1},\dots ,x_{k})$ is a quantifier-free formula such that all of its atomic subformulas $Rx_{i_{1}}\dots x_{i_{k}} $ distinct from identities have 
\begin{equation*}
\{x_{i_{1}}\dots x_{i_{k}}\}=\{x_{1},\dots ,x_{k}\}.
\end{equation*}
In the case where $k=1$, any formula $\forall x\,\phi (x)$ where $\phi $ is quantifier-free, is parametric. For example, $\forall x\,\lnot Rxx\wedge \forall ^{\not=}x,y\,(Rxy\rightarrow Ryx)$ is parametric, whereas $\forall^{\not=}x,y,z\,(Rxy\wedge Ryz\rightarrow Rxz)$ is not. It is not difficult to check that any universal formula of the form $\forall x_{1}\dots \forall x_{k}\,\phi (x_{1},\dots ,x_{k})$ where each atomic subformula of $\phi$ distinct from identities contains all the variables $x_{1},\dots ,x_{k}$ is parametric. A class of structures is called an \emph{Oberschelp class} if it can be axiomatized by parametric sentences.

Oberschelp's extension~\cite[Thm.\ 3]{ober} of Fagin's zero-one law says:\emph{\
on finite models and finite purely relational vocabularies, for any class $\mathbb{K}$
definable by a finite set of parametric sentences, any first-order sentence $\phi$ is almost surely satisfied by finite members of\/ $\mathbb{K}$ or almost surely not satisfied}. More precisely, if $\mathbb{K}_{n}$ is the set of models in $\mathbb{K}$ with domain $[n]=\{1,\ldots,n\},$ then 
\begin{equation*}
\mu _{n}(\phi )=\frac{|\{\mathfrak{M}\in \mathbb{K}_{n}: \mathfrak{M\models } \phi \}|}{|\mathbb{K}_{n}|}\text{ \ \emph{converges to} }0\text{ \emph{or} }1.
\end{equation*}
%and the same result is obtained if we compute the probability in the class
%of unlabeled structures $K_{n}/\approx.$ 
Oberschelp's result actually holds for sentences of the variable-bounded infinitary logic $\mathcal{L}_{\infty\omega }^{\omega}$, i.e.\ the fragment of the infinitary logic $\mathcal{L}_{\infty\omega}$ determined by formulas containing finitely many variables only. This is a generalization of a well-known result by Kolaitis and Vardi for the class of all finite models \cite{KoVa}. An accessible presentation of these results can be found in~\cite{flum}.

If there are no Oberschelp conditions, the counting probability measure $\mu _{n}$ in the above zero-one law corresponds to assigning probability $\frac{1}{2}$ to the atomic events $R(i_{1},\ldots,i_{k})$, $i_{1},\ldots,i_{k}\in \lbrack n]$, in the random model $\mathfrak{M}\in \mathbb{K}_{n}$ and these events are probabilistically independent. The zero-one law holds, with the same asymptotic probability for each $\phi$, if the random model is obtained by assigning a fixed probability $p_{R}\in (0,1)$ to the atomic event $R(x_{1},\ldots,x_{k})$ for each $R\in \tau $.\footnote{In the presence of Oberschelp conditions, one must take conditional probabilities, thus the atomic events are not necessarily independent and their probabilities may differ (but remain independent of the size of the universe).} There is a wealth of positive and negative results pertaining to the case in which $p_{R}$ varies with $n$ (cf.~\cite{spen}).

%An important warning is that the limit probability on formulas $\mu (\phi )=\lim_{n}\mu _{n}(\phi )$ is only finitely additive, not 
%$\sigma $-additive.

\section{The case of finitely valued predicate logics}\label{s:finitely}

Let $\alg{A}=\tuple{A,\wedge^\alg{A},\vee^\alg{A},\ldots}$ be an algebra (which  can be finite or infinite) with a lattice reduct $\tuple{A,\wedge^\alg{A},\vee^\alg{A}}$ and possibly with additional operations (usually $\lnot^\alg{A}$, $\rightarrow^\alg{A}$, or $\conj^\alg{A}$). We will call $\alg{A}$ a \emph{lattice algebra}. The underlying lattice may be bounded, in which case $0$ and $1$ will denote respectively the bottom and top element (which may, but need not, be denoted by constants of the signature of $\alg{A}$).

For the associated $\alg{A}$-valued logic $\mathcal{L}^\alg{A}$ and a first-order
relational vocabulary $\tau $, the first-order logic $\mathcal{L}^\alg{A}(\tau)$ is built in the usual way from variables $x_{1},x_{2},\ldots$, relation symbols in $\tau $, connectives $\wedge, \vee, \ldots$ (corresponding to the operations of $\alg{A}$), and quantifiers $\forall$ and $\exists$.

Models are $\alg{A}$-valued $\tau$-structures $\mathfrak{M}=\tuple{M,R^{\mathfrak{M}}}_{R\in \tau }$ with domain $M\not=\emptyset$, where $R^{\mathfrak{M}}\colon M^{\tau (R)}\rightarrow A$, and $\tau(R)$ denotes the arity of $R\in \tau$; $\Mod_{\tau }^{\alg{A}}$ will denote the class of all $\alg{A}$-valued $\tau$-structures.

Define inductively for each formula $\varphi (x_{1},\ldots,x_{k})$ in $\mathcal{L}^\alg{A}(\tau)$ with $k$ free variables, a mapping $||\varphi ||^{\mathfrak{M}}\colon M^{k}\rightarrow A,$ interpreting the atomic formulas according to $\mathfrak{M}$, the connectives according to $\alg{A}$, and the quantifiers as suprema and infima\ in $\alg{A}$ ($\mathfrak{M}$ is supposed to be {\em safe}, that
is, the required infima and suprema exist, which in the context of finite structures is always guaranteed).
% if either $A$ or $M$ is finite.
% For sentences, $||\varphi ||^{\mathfrak{M}}$ will be an element of $A$.

For a finite (or complete) algebra $\alg{A}$, we may extend the semantics to the infinitary logic $\mathcal{L}_{\omega _{1}\omega }^\alg{A}$ by setting $||\bigwedge\nolimits_{n}\varphi _{n}||^{\mathfrak{M}}:=\inf_{n}||\varphi_{n}||^{\mathfrak{M}}$ (all formulas $\varphi _{n}$ having free variables in a common finite set) and similarly for $\bigvee_{n}\varphi _{n}$.

The classical logics $\mathcal{L}_{\omega \omega}$ and $\mathcal{L}_{\omega_{1}\omega}$ may be respectively identified with $\mathcal{L}^{\alg{B}_{2}}$ and $\mathcal{L}_{\omega _{1}\omega }^{\alg{B}_{2}}$, where $\alg{B}_{2}$ is the two-element Boolean algebra.

\subsection{The zero-one law for finitely valued predicate logics}\label{ss:01law-finitely}

Let $\alg{A}$ be a  finite lattice algebra, $\tau \,$a finite relational vocabulary, and $\Mod_{\tau ,n}^\alg{A}$ be the class of $\alg{A}$-valued  $\tau $-structures with domain $[n]=\{1,\ldots,n\}$.\footnote{This may be identified with the product $\prod_{R\in \tau }A^{[n]^{\tau(R)}}$.
%so that it has power $|A|^{(n + 1)^{\sum _{R}\tau (R)}}.$
} Given $\varphi (x_{1},\ldots,x_{k})\in \mathcal{L}^\alg{A}(\tau )$, $i_{1},\ldots,i_{k}\leq n,$ and $a\in A$, define the probability that a randomly chosen model $\mathfrak{M} \in\Mod_{\tau ,n}^\alg{A}$ assigns the value $a$ to $\varphi (i_{1},\ldots,i_{k})$ as:

\begin{equation*}
\mu _{n}(\varphi (i_{1},\ldots,i_{k})=a):=\frac{|\{\mathfrak{M}\in \Mod_{\tau,n}^\alg{A}: ||\varphi ||^{\mathfrak{M}}(i_{1},\ldots,i_{k})=a\}|}{|\Mod_{\tau ,n}^\alg{A}|}
\end{equation*}
This definition assigns the same weight to all models in $\Mod_{\tau ,n}^\alg{A},$ and the same probability 
\begin{equation*}
\mu _{n}(R(i_{1},\ldots,i_{\tau (R)})=a):=\frac{1}{|A|}
\end{equation*}
to the event that an atomic statement $R(i_{1},\ldots,i_{\tau (R)})$ takes the value $a\in A,$ for any $a$ and numbers $i_{1},\ldots,i_{\tau (R)}\in \lbrack n]$.

More generally, and reciprocally, we may fix, for each $R\in \tau $ a probability distribution $p_{R} \colon A\rightarrow \lbrack 0,1]$ (that is, $\sum_{a\in A}p_{R}(a)=1$) such that $p_{R}(a)>0$ for any $a\in A$, and define a probability distribution Pr$_{n}\colon \Mod_{\tau ,n}^\alg{A}\rightarrow \lbrack 0,1]$ in the class of \alg{A}-valued models with domain $[n]$:

\begin{equation*}
\text{Pr}_{n}(\mathfrak{M}):=\prod\{p_{R}(R^{\mathfrak{M}}(i_{1},\ldots,i_{\tau (R)})) : R\in \tau ,\ i_{1},\ldots,i_{\tau (R)}\in \lbrack n]\}
\end{equation*}
obtained by identifying $\mathfrak{M}$ with its atomic $\alg{A}$-valued diagram.
Then, we may define for any formula $\varphi (x_{1},\ldots,x_{k})\in \mathcal{L}^\alg{A}(\tau ),$ $a\in A,$ and $i_{1},\ldots,i_{k}\in \lbrack n],$
$$\mu _{n}(\varphi (i_{1},\ldots,i_{k})=a):=\sum \{Pr_{n}(\mathfrak{M}) : \mathfrak{M} \in \Mod_{\tau ,n}^\alg{A}, ||\varphi ||^{\mathfrak{M}}(i_{1},\ldots,i_{k})=a\}.$$ 

Notice that, for the sake of clarity, we are using the simplified notation $\varphi (i_{1},\ldots,i_{k})=a$ to denote the event $ ||\varphi ||^{\mathfrak{M}}(i_{1},\ldots,i_{k})=a$ for an arbitrary model $\mathfrak{M}\in \Mod_{\tau ,n}^\alg{A}$.

It is not hard to check that, under this definition, 
\begin{equation*}
\mu _{n}(R(i_{1},\ldots,i_{\tau (R)})=a)=p_{R}(a),
\end{equation*}
and the atomic events $R(i_{1},\ldots,i_{k})=a,$ $R^{\prime }(j_{1},\ldots,j_{k})=b$
are (probabilistically) independent if $R \neq R'$ or $\tuple{i_{1},\ldots,i_{k}} \neq \tuple{j_{1},\ldots,j_{k}}$ or $a \neq b$.

The next theorem generalizes Fagin's zero-one law to the finitely-valued context.

\begin{Thm}\label{label} For any sentence $\varphi \in \mathcal{L}^\alg{A}(\tau )$ and $a\in A$, $\lim_{n}\mu _{n}(\varphi =a)=1$ or $\lim_{n}\mu _{n}(\varphi
=a)=0$. Equivalently, for any sentence $\varphi \in \mathcal{L}^\alg{A}(\tau )$, there is a (unique) $a\in A$ such that $\lim_{n}\mu _{n}(\varphi =a)=1$.
\end{Thm}

\begin{proof} The equivalence of both statements follows from the fact that, for each $\varphi$, the events $\varphi=a$, with $a\in A$, are exhaustive and mutually incompatible. The first formulation follows from the classical zero-one laws via the following multi-translation $\theta \longmapsto \tuple{\theta^{a}}_{a\in A}$, defined by simultaneous induction, from $\mathcal{L}^\alg{A}(\tau)$ into classical $\mathcal{L}_{\omega \omega }(\tau \times A)$, $\tau\times A=\{R^{a}:R\in \tau ,a\in A\}$:

$R(x_{1},\ldots,x_{\tau (R)})^{a}:=R^{a}(x_{1},\ldots,x_{\tau (R)})$, for $R\in \tau$

$\circ(\varphi _{1},\ldots,\varphi_{k})^{a}:=\bigvee_{\circ^\alg{A}(b_{1},\ldots,b_{k})=a}\varphi _{1}^{b_{1}}\wedge\ldots\wedge \varphi _{k}^{b_{k}}$, for any connective $\circ$ of arity $k$ in the
signature of $\alg{A}$

$(\forall x\varphi )^{a}:=(\bigvee_{b_{1}\wedge \ldots\wedge b_{m}=a,\text{ }m\leq |A|}\exists x\,\varphi^{b_{1}}\wedge \ldots\wedge \exists x\,\varphi^{b_{m}})\wedge \forall x\,(\bigvee_{b\geq a}\varphi ^{b})$

$(\exists x\,\varphi )^{a}:=(\bigvee_{b_{1}\vee \ldots\vee b_{m}=a,\text{ }m\leq|A|}\exists x\,\varphi^{b_{1}}\wedge \ldots\wedge \exists x\, \varphi^{b_{m}})\wedge \forall x\,(\bigvee_{b\leq a}\varphi ^{b})$,

\noindent and the corresponding model transformations from $\Mod_{\tau }^\alg{A}$
into $\Mod_{\tau \times A}$:
\begin{equation*}
\mathfrak{M}\longmapsto \mathfrak{M}_{\tau \times A}:=\tuple{M;(R^{a})^{\mathfrak{M}_{\tau \times A}}}_{R^{a}\in \tau \times A},
\end{equation*}
where $(R^{a})^{\mathfrak{M}_{\tau \times A}}:=\{\overline{i}\in M : R^{\mathfrak{M}}(\overline{i})=a\}$,
which are easily seen to satisfy:
\begin{equation*}
||\theta||^{\mathfrak{M}}(i_{1},\ldots,i_{k})=a \ \ \Leftrightarrow \ \mathfrak{M}_{\tau \times A}\models \theta^{a}[i_{1},\ldots,i_{k}].
\end{equation*}
\label{transl}The model transformations establish a domain-preserving bijection between $\Mod_{\tau}^\alg{A}$ and the class $\mathbb{K}\subseteq \Mod_{\tau\times A}$ of classical models of the parametric $(\tau \times A)$-sentences

$\forall x_{1}\ldots x_{\tau (R)}\,\bigvee_{a\in A}R^{a}(x_{1},\ldots,x_{\tau (R)}),$

$\forall x_{1}\ldots x_{\tau (R)}\,\bigwedge_{a\not=b}\lnot (R^{a}(x_{1},\ldots,x_{\tau(R)})\wedge R^{b}(x_{1},\ldots,x_{\tau (R)})).$

\noindent Being an Oberschelp class, $\mathbb{K}$ satisfies the classical zero-one law if we assign probability $\mu _{n}^{\mathbb{K}}(R^{a}(x_{1},\ldots,x_{k})):=p_{R}(a)$ to the classical atomic events $R^{a}(x_{1},\ldots,x_{k})$. As the latter assignment yields $\Pr_{n}(\mathfrak{M}_{\tau \times A})=\Pr_{n}(\mathfrak{M)}$ for $\mathfrak{M} \in \Mod_{\tau ,n}^\alg{A}$ and the transformation is bijective  from $\Mod_{\tau ,n}^\alg{A}$ onto $\mathbb{K}_n$ (models in $\mathbb{K}$ with domain $[n]$). Then,
$\mu_{n}(\varphi =a)=\\\mu _{n}\{\mathfrak{M}\in \Mod_{\tau ,n}^\alg{A} : ||\varphi||^{\mathfrak{M}}=a\}=\mu _{n}^{\mathbb{K}}\{\mathfrak{M}_{\tau \times A}\in \mathbb{K}_{n} : \mathfrak{M}_{\tau \times A}\models \varphi^{a}\}=\mu _{n}^{\mathbb{K}}(\varphi^{a})$.

Thus, $\mu _{n}(\varphi =a)$ converges to $0$ or $1$.
\end{proof}

\begin{Rmk} \label{rmk1} \emph{Theorem \ref{label} holds if the models are restricted to any class of $\alg{A}$-valued structures satisfying conditions that are parametrically expressible in $\mathcal{L}_{\omega \omega }(\tau \times A)$ via the translation. Just add these conditions to the definition of the class $\mathbb{K}$ in the proof.
For example, $\alg{A}$-valued irreflexive symmetric graphs:}

$\forall x\,R^{0}xx$, $\forall ^{\not=}xy\,\bigwedge_{a\in A}(R^{a}xy\leftrightarrow R^{a}yx),$

\noindent\emph{or structures with crisp identity:}

$\forall x\,(x\approx ^{1}x)$, $\forall ^{\not=}xy\,(x\approx ^{0}y)$.\end{Rmk}

\begin{Rmk} \emph{Theorem~\ref{label} holds also for sentences in finite vocabularies of the variable-bounded sublogic $(\mathcal{L}^\alg{A})_{\omega_{1}\omega }^{\omega }$ of $\mathcal{L}_{\omega _{1}\omega}^\alg{A}$, where each formula has finitely many free and bound variables, because of the obvious extension of the translation to the infinitary logic; for example, $$(\bigwedge\nolimits_{n}\varphi _{n})^{a}:=\bigvee_{b_{1}\wedge \ldots\wedge b_{k}=a,\text{ }k\leq |A|,\text{ }n_{1},\ldots,n_{k}}(\varphi_{n_{1}}^{b_{1}}\wedge \ldots\wedge \varphi _{n_{k}}^{b_{m}})\wedge \bigwedge\nolimits_{n}\bigvee _{b\geq a}\varphi _{n}^{b}$$ sends it to classical $\mathcal{L}_{\omega _{1}\omega }^{\omega}$, known to satisfy the zero-one law~\cite{KoVa}).}\end{Rmk}

\begin{Rmk}\label{rem:S5}\emph{Zero-one laws have been studied for classical modal logics (e.g.~\cite{kapron1,lebars}). In the finitely valued case, any fully modal sentence of the $\alg{A}$-valued version S$5_{c}^\alg{A}$ of propositional bimodal logic S$5$ over finite universal crisp frames (cf.~\cite{Caicedo-Rodriguez:OrderModal,Caicedo-Rodriguez:BimodalGodel}) satisfies a
zero-one law, because this logic may be interpreted as the one-variable fragment of $\mathcal{L}^\alg{A}$ on unary predicates via the translation $\varphi \mapsto \widehat{\varphi}$:}

$\widehat{p_{i}}:=P_{i}(w)$

$\widehat{\square \varphi}:=\forall w\widehat{\varphi }(w)$

$\widehat{\Diamond \varphi}:=\exists w\widehat{\varphi }(w)$

$\widehat{t(\varphi _{1},\ldots,\varphi _{k})}:=t(\widehat{\varphi }_{1}(w),\ldots,\widehat{\varphi }_{k}(w)),$

\noindent \emph{which sends any fully modal formula to a sentence.}
\end{Rmk}

%$\emph{[[[}$3. the positive classical results on convergence for variable probabilities $p_{R,n}$ transfer automatically to the finitely valued context.]]]

\subsection{Finitely valued random models and extension axioms}\label{ss:finitely-random}

In the above proof of the finitely valued zero-one law lurks a countable $\alg{A}$-valued random model corresponding the classical random model that one can obtain for the language of the translation.

\begin{Thm}\label{count}
There is a countable $\alg{A}$-valued $\tau $-model $\mathfrak{R}_{\tau }^\alg{A}$ such that, for any sentence $\varphi \in \mathcal{L}^\alg{A}(\tau)$, we have: $\lim_n \mu_{n}(\varphi =a)= 1$ iff\/ $||\varphi ||^{\mathfrak{R}_{\tau}^\alg{A}}=a $.
\end{Thm}

\begin{proof}
Take the $\alg{A}$-model $\mathfrak{R}_{\tau }^\alg{A}$
corresponding to the countable $\omega $-categorical random model $\mathfrak{R}_{\tau \times A}$ of the class $\mathbb{K}$ introduced in the proof of Theorem \ref{label}. Then, $\lim_n\mu _{n}(\varphi =a)= 1$ iff $\lim_n \mu _{n}^{\mathbb{K}}(\varphi ^{a})= 1$ iff $\mathfrak{R}_{\tau \times A}\models \varphi ^{a}$ iff $||\varphi ||^{\mathfrak{R}_{\tau}^\alg{A}}=a$.
\end{proof}

The classical extension axioms characterizing $\mathfrak{R}_{\tau \times A}$, and indirectly $\mathfrak{R}^\alg{A}_{\tau}$, are rarely expressible in $\mathcal{L}^\alg{A}$, but they are expressible in ($\mathcal{L}_{\omega \omega })_{\tau \times A}$ as:

\begin{center}
$Ext_{A}(k,f):\forall ^{\not=}x_{1},\ldots,x_{k}\,\exists x_{k+1}\not=x_{1},\ldots,x_{k}\,\bigwedge\nolimits_{\varphi \in F_{k+1}}\varphi^{f(\varphi )}(x_{1},\ldots,x_{k+1})$,
\end{center}

\noindent where $F_{k+1}$ is the set of atomic $\tau$-formulas with variables in $x_{1},\ldots,x_{k},x_{k+1}$, truly containing the variable $x_{k+1}$, and $f\colon F_{k+1}\rightarrow A$. These axioms are automatically consistent with the set of parametric sentences defining the class $\mathbb{K}$ in the proof of Theorem~\ref{label}. If there are additional parametric conditions, only the extension axioms consistent with them are considered.

\begin{Exm}\emph{According to Remark \ref{rmk1}, the random $\alg{A}$-valued irreflexive symmetric graph may be described as the unique countable $\alg{A}$-weighted irreflexive symmetric graph $\mathfrak{R}=\tuple{G,R}$ such that, for any distinct $x_{1},\ldots,x_{k}\in G$ and $a_1,\ldots, a_k\in A$, there is an $x_{k+1}\not=x_{1},\ldots,x_{k}$ in $G$ such that $R(x_{i},x_{k+1})=a_i$. This structure was studied as a Fra\"iss\'e limit in \cite{BaNog18}.}
\end{Exm}

%From classical finite model theory, $\lim_{n}\mu_{n}(Ext_{A}(k,f))\rightarrow 1,$ [[[if $p_{R}(f(\varphi ))>0$ ]] which implies the consistency, $\omega $-categoricity, and completeness of the
%theory $Ext_{A}=\mathit{Disj}(\tau \times A)\{Ext_{A}(k,f):k\geq 1,f\in A^{F_{k+1}}\},$ so that $\mathfrak{R}_{\tau \times A}\models \varphi ^{a}$ iff $Ext_{A}\models \theta $ iff $\mu _{n}^{\mathcal{\mathbb{K}}}(\theta)\rightarrow 1$.
Since the extension axioms and thus the random model are independent of the probabilities $\{p_{R}\}_{R\in \tau}$, we can use Theorem~\ref{count} to obtain the following corollary.

\begin{Cor}\label{ob}
For any sentence $\varphi \in \mathcal{L}^\alg{A}$, the value $a\in A$ such that $\lim_n \mu _{n}(\varphi =a)= 1$ depends only on $\varphi $, and not on the probability distributions $p_{R}$ (provided that $p_{R}(a)>0$ for all $a\in A)$.
\end{Cor}

We call the value $a$ in Corollary~\ref{ob} the \emph{almost sure} or \emph{asymptotic} truth-value of $\varphi$. This corollary entails that  assigning the same probability $\frac{1}{|A|}$ to all truth-values has the same effect in the limit as assigning a positive probability to each one.

\begin{Rmk}\emph{If $p_{R}$ is allowed to take value 0 at some $a\in A$, this corresponds to asking the Oberschelp conditions $\forall x_{1}\ldots\forall x_{n}\,\lnot R^{a}(x_{1},\ldots,x_{n})$. Thus, the asympotic value of $\varphi$ still exists and depends on the support of $p_{R}$ only.}
\end{Rmk}

%\guillermo{[[Otro fen\'{o}meno no visible en el caso caso cl\'{a}sico: por ejemplo, $p_{R}(1_{A})=0$ produce grafos discretos$:\forall xy\lnot R(x,y)$, cuya teor\'{\i}a asint\'{o}tica es trivial. De todas maneras, $1_{A}$ sigue siendo un valor asimpt\'{o}tico pues $\mu _{n}(\forall xy\lnot R(x,y))=1.$ An\'{a}logamente, $p_{R}(0_{A})=0,$ produce grafos completos, etc.]]}

\subsection{Complexity of almost sure values}\label{ss:complexity}

Grandjean~\cite{grand} showed that the problem of determining for the classical zero-one law if a sentence is almost surely true or not is in PSPACE, and it has been observed~\cite{com} that this holds in the presence of Oberschelp conditions. Moreover, it is PSPACE-complete, even for the vocabulary of pure identity.

Since our translation is exponential in the length of the original formula, it yields only an exponential space upper bound for the computation of the almost sure value of an $\alg{A}$-valued sentence. However, the shape of the translation will allow us to give a simple direct proof of an upper polynomial space bound. PSPACE-completeness will be achieved for logics with crisp identity.

\begin{Thm}
For any sentence $\varphi \in \mathcal{L}^\alg{A}(\tau)$, the almost sure value of $\varphi$ may be computed in polynomial space.
\end{Thm}

We denote by $Ext_A$ the set of all extension axioms. It is enough to show that for any $a\in A$, it may be determined in polynomial space whether $Ext_A\vdash \varphi ^{a}$. Then, one may check this for each $a\in A$ utilizing the same space until finding the (unique) $a$ for which it holds. Using Grandjean's strategy~\cite{grand}, we can do it without writing $\varphi^{a}$ explicitly.

Given a sequence of variables $\mathbf{x}=x_{1},\ldots,x_{k}$, a \emph{complete
description} of $\mathbf{x}$ is a $(\tau \times A)$-formula $\Delta (\mathbf{x)}$ of the form $\bigwedge\nolimits_{\psi \in At_{k}}{\psi}^{f(\psi)}$ where $At_{k}$ is the set of atomic $\tau$-formulas with variables in $\mathbf{x}$ and $f \colon At_{k} \to A$. It is easy to see that any such $\Delta (\mathbf{x)}$ is realizable in the classical random model $\mathfrak{R}_{\tau \times A}$, because existential sentences are always almost surely true.

Notice that, due to the strong homogeneity of $\mathfrak{R}_{\tau \times A}$
(any isomorphism between finite substructures may be extended to an
automorphism, see \cite{Hodges}), all tuples of $\mathfrak{R}_{\tau \times A}$ satisfying the
same complete description $\Delta (\mathbf{x)\ }$have the same type in $\mathfrak{R}_{\tau \times A}$; therefore, the theory $Ext_{A}+\Delta (\mathbf{x)}$ is complete for all formulas with their free variables in $\mathbf{x}$.

In what follows, $|\alpha |$ will denote the length of a formula or a sequence of variables, or the cardinality of a finite set. For simplicity we assume that the symbols in $\tau $ and $\tau \times A$ as well as the variables have length $1$. Particularly, $|\theta ^{a}|=|\theta |$ for any atomic $\theta \in L_{\tau }^\alg{A}$ and $a\in A$.

Set $M_{\tau }=\max \{\tau (R):R\in \tau \}$. Then, it is easy to check that $|\Delta (\mathbf{x)|=}\sum _{R\in \tau }(\tau (R)+c^{\prime })|\mathbf{x}|^{\tau (R)}\leq |\tau |(M_{\tau}+c^{\prime })|\mathbf{x}|^{M_{\tau}}=c|\mathbf{x}|^{M_{\tau}}$ for a suitable constant $c$.

\begin{Lem}
Given a complete description $\Delta(\mathbf{x})$, any formula $\theta \in \mathcal{L}_{\tau}^\alg{A}$ with its free variables contained in $\mathbf{x}$ and $a\in A$, it may be decided whether $Ext_{A}+\Delta (\mathbf{x})\vdash \theta^{a}$ or not in space $c(|\mathbf{x}| + |\theta |)^{M_{\tau}}$.
\end{Lem}

\begin{proof}
Induction on the complexity of $\theta$.

\begin{enumerate}
\item If $\theta$ is atomic, then $Ext_{A}+\Delta (\mathbf{x})\vdash \theta^{a}$ iff $\theta^{a}$ is a conjunct of $\Delta (\mathbf{x),}$ and this is clearly verifiable in space $|\Delta (\mathbf{x)}|+|\theta ^{a}|=|\Delta (\mathbf{x)}|+|\theta |\leq c|\mathbf{x}|^{M_{\tau}}+|\theta |\leq c(|\mathbf{x}|+|\theta |)^{M_{\tau}}$.

\item If $\theta $ is $\circ(\varphi _{1},\ldots,\varphi _{k})$ where $\circ$ is a connective,
then by definition of the translation and completeness of $Ext_{A} + \Delta (\mathbf{x})$: $Ext_{A} + \Delta (\mathbf{x})\vdash \theta^{a}$ iff there are $b_{1},\ldots,b_{k}\in A$ such that $\circ^\alg{A}(b_{1},\ldots,b_{k})=a$ and
\begin{equation*}
Ext_{A}+\Delta (\mathbf{x})\vdash \varphi _{i}^{b_{i}},\text{ for } i=1,\ldots,k.
\end{equation*}
But the $k$ statements in the display may be verified each one in space $ c(|\mathbf{x}|+|\theta
|)^{M_{\tau}} \geq c(|\mathbf{x}|+|\varphi _{i}|)^{M_{\tau}}$ by induction hypothesis. And we may do this sequentially in
the same space  for each $b_{i}$ of
each relevant tuple $b_{1},\ldots,b_{k}$, until finding a tuple for which the
displayed statement holds.

\item If $\theta $ is $\forall x\,\varphi (x)$, then by definition of the
translation and completeness: $Ext_{A}+\Delta(\mathbf{x})\vdash (\forall
x\,\theta (x))^{a}$ iff there are $b_{1},\ldots,b_{m}\in A$ such that $b_{1}\wedge
\ldots\wedge b_{m}=a$, and 
\begin{equation*}
Ext_{A}+\Delta (\mathbf{x})\vdash \exists x\,\varphi (x)^{b_{1}},\ldots,\exists
x\,\varphi (x)^{b_{m}},\forall x\,(\bigvee_{b\geq a}\varphi (x)^{b})
\end{equation*}
equivalently,
\begin{eqnarray*}
Ext_{A}+\Delta (\mathbf{x}) &\vdash &\exists x\,\varphi (x)^{b_{i}},\text{ for 
}i=1,\ldots,m \\
Ext_{A}+\Delta (\mathbf{x}) &\not\vdash &\exists x\,\varphi (x)^{b},\text{ for
all }b\not\geq a
\end{eqnarray*}
because $\forall x\,(\bigvee_{b\geq a}\varphi (x)^{b})$ is equivalent in $Ext_{A}$ to $\bigwedge\nolimits_{b\not\geq a}\lnot \exists x\,\varphi (x)^{b}$, and $Ext_{A}+\Delta (\mathbf{x})\vdash \lnot \exists x\,\varphi (x)^{b}$ iff 
$Ext_{A}+\Delta (\mathbf{x})\not\vdash \exists x\,\varphi (x)^{b}$ due to
completeness. Now, $Ext_{A}+\Delta (\mathbf{x})\vdash \exists x\,\varphi
(x)^{b_{1}}$ if and only if there is an expansion $\Delta ^{\prime }(x,\mathbf{x})$ of $\Delta (\mathbf{x})$ such that 
\begin{equation*}
Ext_{A}+\Delta ^{\prime }(x,\mathbf{x})\vdash \varphi (x)^{b_{1}}
\end{equation*}
(if $\mathbf{a}$ is a tuple in $\mathfrak{R}_{\tau \times A}$ satisfying $\Delta (\mathbf{x})$ and $b$ is a witness of $\exists x\,\varphi (x,\mathbf{a})^{b_{1}}$, then the complete description $\Delta ^{\prime }(x,\mathbf{x})$
of $\mathfrak{R}_{\tau \times A}\upharpoonright \{b,\mathbf{a}\}$ does the
job$)$. By induction hypothesis, we may check this for each possible
expansion of $\Delta (\mathbf{x})$ in space $c(|\mathbf{x}|+1+|\varphi
(x)|)^{M_{\tau}}\leq c(|\mathbf{x}|+|\theta |)^{M_{\tau}}$. Doing this
sequentially in the same space for each $b_{1},\ldots,b_{m}, b$ such that $b_{1}\wedge \ldots\wedge b_{m}=a$ and $b\not\geq a$, we may check whether $Ext_{A}+\Delta (\mathbf{x})\vdash (\forall x\,\theta (x))^{a}$.

\item If $\theta $ is $\exists x\,\varphi (x)$, then, similarly to the previous case, $Ext_{A}+\Delta (\mathbf{x})\vdash (\exists x\,\varphi (x))^{a}$ iff there are $b_{1},\ldots,b_{m}\in A$ such that $b_{1}\vee \ldots \vee b_{m}=a$, and
\begin{eqnarray*}
Ext_{A}+\Delta (\mathbf{x}) &\vdash &\exists x\,\varphi (x)^{b_{i}},\text{ for 
}i=1,\dots,m \\
Ext_{A}+\Delta (\mathbf{x}) &\not\vdash &\exists x\,\varphi (x)^{b},\text{ for
all }b\not\leq a
\end{eqnarray*}
and this claim may be verified in space $c(|\mathbf{x}|+|\theta |)^{M_{\tau}}$.
\end{enumerate}
\end{proof}

\begin{Cor}
The almost sure value of any sentence $\varphi \in \mathcal{L}^\alg{A}(\tau)$ may be computed in space $c(1+|\varphi |)^{M_{\tau}}$.
\end{Cor}

\begin{proof}
Let $\Delta (x)$ be any complete description in one variable. Then, $Ext_{A}\vdash \varphi ^{a}$ iff $Ext_{A}+\exists x\Delta
(x)\vdash \varphi ^{a}$ (because $Ext_{A}\vdash \exists x\,\Delta (x))$ iff $Ext_{A}+\Delta (x)\vdash \varphi ^{a}$. Thus, $Ext_{A}\vdash \varphi ^{a}$ may be decided in space $c(1+|\varphi |)^{M_{\tau}}$. Check this for any $a\in A$ until finding the only $a$ for which it holds.
\end{proof}

If the logic satisfies additional Oberschelp conditions in the translation (besides the necessary ones), the previous result is obtained by considering only complete descriptions compatible with those conditions.

\emph{$\alg{A}$-valued logic with crisp identity}, in symbols $\mathcal{L}_{\approx}^\alg{A}$, will be $\mathcal{L}^\alg{A}$ enriched with a distinguished binary relation symbol $\approx$ interpreted in each model as the characteristic function of the diagonal.

This corresponds to having the Oberschelp conditions $\forall x\,(x\approx^{1}x)$, $\forall^{\not=}xy\,(x\approx^{0}y)$ in the translated logic. Hence, this logic has a zero-one law and satisfies the above results.

\begin{Thm}
If $\alg{A}$ has a term $u$ such that $u^\alg{A}(1)=0$ and $u^\alg{A}(0)=1$, then computing the almost sure values of sentences of $\mathcal{L}_{\approx}^\alg{A}$ is PSPACE-complete.
\end{Thm}

\begin{proof} Consider the translation $\varphi \longmapsto \widehat{\varphi}$ from classical pure identity logic into $\mathcal{L}_{\approx}^\alg{A}$ that sends $=$ to $\approx $ and $\lnot $ to $u$. A straightforward induction on the complexity of $\varphi $ shows that, for any $\mathbf{m}$ in $\mathfrak{R}_{\tau }^\alg{A}$, $||\widehat{\varphi}||^{\mathfrak{R}_{\tau}^\alg{A}}(\mathbf{m})\in \{0,1\}$ and $||\widehat{\varphi}||^{\mathfrak{R}_{\tau}^\alg{A}}(\mathbf{m})=1$ iff $\mathfrak{R}_{\tau \times A}\models \varphi \lbrack \mathbf{m}]$.  Notice that for the classical random model of the extension axioms of this logic: $\mathfrak{R}_{\tau \times A}\models m\approx^{1}n$ iff $m=n$, $\mathfrak{R}_{\tau \times A}\models m\approx ^{0}n\,$ iff $m\not=n$. Therefore, for any classical identity sentence $\varphi $, $\lim_n\mu _{n}(\widehat{\varphi}=1)= 1$ iff $||\widehat{\varphi}||^{\mathfrak{R}_{\tau}^\alg{A}}=1$ iff $Ext_{A}\vdash \varphi $ iff $I_{\infty}\vdash \varphi $, where $I_{\infty}$ is the theory of identity in infinite sets. But deducibility in the later theory is PSPACE-hard by results of Stockmeyer~\cite{stock}.
\end{proof}

\section{The set of almost sure values in finitely valued predicate logics}\label{s:SetAlmostSure}

Call $a\in A$ an \emph{almost sure} truth-value of $\mathcal{L}^\alg{A}$ (of $\alg{A}$)  if there is a sentence $\varphi \in \mathcal{L}^\alg{A}$ such that $\lim_n\mu_{n}(\varphi =a)=1$. Do all the elements of $A$ arise as almost sure values? If not, which are those values? We will see that for
many familiar classes of finite algebras the answer to the first question is positive (for example \MV-chains for \L ukasiewicz logic, or \G-chains for G\"odel logic), for others only a small subset are almost sure values ($0$ or $1$ for Boolean algebras, at most four values for De Morgan algebras, at most three for the lattice-negation reducts of G-chains, etc.).

The first assertion follows from the following general lemma.

\begin{Lem}\label{al}
Let $t(v_{1},\ldots,v_{k})$ be a term in the signature of a finite lattice algebra $\alg{A}$, then $m=\inf_{a_{1},\ldots,a_{k}\in A}t^\alg{A}(a_{1},\ldots,a_{k})$ and $s=\sup_{a_{1},\ldots,a_{k}\in A}t^\alg{A}(a_{1},\ldots,a_{k})$ are almost sure truth-values of $\mathcal{L}^\alg{A}$.
\end{Lem}

\begin{proof}
Let $R(x)$ be an atomic formula with a single variable $x$.  We show that
\begin{equation*}
\mu _{n}(\forall x_{1}\ldots\forall x_{k}\,t(R(x_{1}),\ldots,R(x_{k}))=m)\ \text{converges to}\ 1.
\end{equation*}
Fix tuples $\tuple{a_{j}^{1}}_{j},\ldots,\tuple{a_{j}^{\ell}}_{j}\in |A|^{k}$, such that $t^\alg{A}(a_{1}^{1},\ldots,a_{k}^{1})\wedge^\alg{A}\ldots\wedge ^\alg{A}t^\alg{A}(a_{1}^{\ell},\ldots,a_{k}^{\ell})=m,$ and\ for each $n\geq k\ell $ find $K\subseteq\lbrack n]^{k}$ of power $|K|=\lfloor n/k\ell \rfloor $ such that all tuples in $K$ have all their components distinct, and two different tuples in $K$ are distinct at each coordinate. Then,

$\mu _{n}(\forall x_{1}\ldots\forall x_{k}\,t(R(x_{1}),\ldots,R(x_{k}))=m)$

$=\Pr_{n}(\inf_{i_{1},\ldots,i_{k}\in \lbrack n]}t^\alg{A}(R(i_{1}),\ldots,R(i_{k}))=m)$

$\geq \Pr_{n}($for some $\tuple{i_{j}^{1}}_{j},\ldots,\tuple{i_{j}^{\ell}}_{j}$ in $[n]^{k}:R(i_{j}^{i})=a_{j}^{i},$ for $i=1,\ldots,\ell ,$ $j=1,\ldots,k)$

$\geq \Pr_{n}($for some $\tuple{i_{j}^{1}}_{j},\ldots,\tuple{i_{j}^{\ell}}_{j}$ in $K:R(i_{j}^{i})=a_{j}^{i},$ for $i=1,\ldots,\ell ,$ $j=1,\ldots,k)$

$=1-\Pr_{n}($for all $\tuple{i_{j}^{1}}_{j},\ldots,\tuple{i_{j}^{\ell}}_{j}$ in $K$ is
false: $R(i_{j}^{i})=a_{j}^{i},$ for $i=1,\ldots,\ell$, $j=1,\ldots,k)$

$=1-\prod_{\tuple{i_{1}^{1},\ldots, i_{k}^{1}},\ldots,\tuple{i_{1}^{\ell},\ldots, i_{k}^{\ell}}\in
K}[1-\Pr_{n}(R(i_{j}^{i})=a_{j}^{i}$ for $i=1,\ldots,\ell$, $j=1,\ldots,k))]$

$=1-(1-p_{R}(a_{1}^{1})\ldots p_{R}(a_{k}^{\ell }))^{\lfloor n/k\ell \rfloor}$.

\noindent For the last two identities we use that the $k\ell $ events $R(i_{j}^{i})=a_{j}^{i}$ form a mutually independent set for distinct $i_{1}^{1},\ldots, i_{k}^{1},\ldots,i_{1}^{\ell},\ldots, i_{k}^{\ell }\in K$. Since $p_{R}(a_{i}^{j})>0$, the last quantity converges to $1$. A similar proof shows
that 
\begin{equation*}
\mu _{n}(\exists x_{1}\ldots\exists x_{k}\,(t(R(x_{1}),\ldots,R(x_{k}))=s) \,\, \text{converges to }1.\qedhere
\end{equation*}
\end{proof}

\begin{Exm}
\emph{It follows that $0$ and $1$ are always almost sure values because, even if they do not appear in the signature of $A$, they are inf and sup of the term $t(v):=v$.}
\end{Exm}

\begin{Exm}\label{ex:LukN}
\emph{Let $\mathrmL_{N+1}=\tuple{\{0,\frac{1}{N},\ldots,\frac{N-1}{N},1\},\wedge ,\vee ,\lnot,\rightarrow}$ be the \MV-chain of $N+1$ elements,
where $\wedge,\vee$, are $\min$ and $\max$, and $\lnot ,\rightarrow$ are \L ukasiewicz implication and negation. Then all values of $\mathrmL_{N+1}$ are almost sure values of the corresponding predicate logic. To see this, notice that using (Chang) connectives:}

\emph{$a\oplus b:=\lnot a\rightarrow b$, $a\odot b:=\lnot (a\rightarrow \lnot b)$, 
$na:=a\oplus \ldots\oplus a,$ $a^{n}:=a\odot \ldots\odot a$ $(n\geq 1)$, the term $t(v)=v^{N}\oplus \lnot v$ has minimum value $\frac{1}{N}$
in $\mathrmL_{N+1}$ (taken at $a=\frac{N-1}{N})$, and thus $t_{k/N}(v)=k(v^{N}\oplus \lnot v)=v^{k}\rightarrow kv^{N} $ has minimum value $\frac{k}{N}$.}
\end{Exm}

\begin{Exm}\label{ex:GN}
\emph{Let $\alg{G}_{N+1}=\tuple{\{0=g_{0}<\ldots<g_{N}=1\},\wedge ,\vee ,\lnot ,\rightarrow}$
be the G-chain of $N+1$ elements, where $\lnot ,\rightarrow $ are G\"{o}del implication and negation, then all values of $\alg{G}_{N+1}$ are almost sure values of the corresponding predicate logic. Define:}

\emph{ $t_{1}(v_{1}):=v_{1}\vee \lnot v_{1}$,
$t_{k+1}(v_{1},\ldots,v_{k+1}):=v_{k+1}\vee (v_{k+1}\rightarrow
t_{k}(v_{1},\ldots,v_{k}))$}

\emph{
\qquad \qquad  $=v_{k+1}\vee (v_{k+1}\rightarrow (v_{k}\vee
(v_{k}\rightarrow \ldots\rightarrow (v_{2}\vee (v_{2}\rightarrow (v_{1}\vee
\lnot v_{1}))\ldots)))$}

\noindent\emph{then the minimum possible value of $t_{1}(v_{1})$ in $\alg{G}_{N+1}$ is 
$g_{1}\vee \lnot g_{1}=g_{1},$ since $g_{0}\vee \lnot g_{0} =1$
and $g_{j}\vee \lnot g_{j}=g_{j}$ for $g_{j}>g_{1}$. For $2\leq k\leq N$: 
\begin{equation*}
\,\min_{v_{1},\ldots,v_{k}\in G_{N+1}}t_{k}(v_{1},\ldots,v_{k})=g_{k}
\end{equation*}%
obtained at $g_{1},\ldots,g_{k},$ because
\begin{center}
$t_{k}(x_{1},\ldots,x_{k})=\left\{ 
\begin{array}{cc}
1 & \text{ if }x_{i+1}\leq x_{i}\text{ for some }i \\ 
x_{k} & \text{ if }x_{i}<x_{i+1}\text{ for all }i.
\end{array}
\right. $
\end{center}}
\end{Exm}

\begin{Rmk}
\emph{Notice that in the proof of Lemma~\ref{al}, we have used a single predicate symbol and $k$ variables to write the sentences}

$\forall x\, t(P_{1}(x),..,P_{k}(x))$

$\exists x\, t(P_{1}(x),..,P_{k}(x))$

\noindent\emph{having the desired almost sure values $m$ and $s$. We may use instead $k$ (unary) predicates and a single variable with a similar proof. Therefore, via the traslation explained in Remark~\ref{rem:S5}, the formulas $\square
t(p_{1},\ldots,p_{k})$ and $\Diamond t(p_{1},\ldots,p_{k})$ of the S$5_{c}^{A}$ have
almost sure values $m$ and $s,$ respectively$.$ It follows then from Example~\ref{ex:LukN} that all values of $\mathrmL_{N+1}$ are almost sure values of S$5_{c}^{\mathrmL_{N+1}},$ since the formula $\square (p^{k}\rightarrow kp^{N})$ has
almost sure value $\frac{k}{N},$ and it follows, from Example~\ref{ex:GN}, that all
values of G$_{N+1}$ are almost sure values of S$5_{c}^{G_{N+1}},$ since the
modal formula}

$\square (p_{k}\vee (p_{k}\rightarrow (p_{k-1}\vee (p_{k-1}\rightarrow
(\ldots(p_{1}\vee \lnot p_{1})\ldots))$

\noindent\emph{takes the almost sure value $g_{k}$.}
\end{Rmk}

We turn now to algebras with few almost sure values and obtain the following result inspired by the techniques in~\cite{gra}.

\begin{Thm}
\label{De Morgan}Let $\alg{A}=\tuple{A,\wedge ,\vee ,\lnot}$ be a finite distributive
lattice with a negation satisfying:
\begin{enumerate}
\item $\lnot (v\wedge w)=\lnot v\vee \lnot w,$ $\lnot (v\vee w)=\lnot
v\wedge \lnot w$
\item $v\leq \lnot \lnot v$
\item $\lnot 1=0$.
\end{enumerate}
Then, the almost sure values of $\mathcal{L}^\alg{A}$ are $0$, $\varepsilon
,\varepsilon ^{\prime },\delta ,\delta ^{\prime}\!$, and\/ $1$, where

\quad $\varepsilon =\sup_{x\in A}(x\wedge \lnot x),$ $\varepsilon ^{\prime
}=\sup_{x\in A}(\lnot x\wedge \lnot \lnot x),\,$

\quad $\delta =\inf_{x\in A}(x\vee \lnot x),$ $\delta ^{\prime
}=\inf_{x\in A}(\lnot x\vee \lnot \lnot x)$.
\end{Thm}

\begin{proof}
We already know these values are almost sure by Lemma~\ref{al}. In order to see that they are the only ones, notice first the following derived laws of $\alg{A}$:
\begin{enumerate}
\setcounter{enumi}{3}
\item $v\leq w$ implies $\lnot w\leq \lnot v$
\item $\lnot \lnot \lnot v=\lnot v$
\item $\lnot 0=\lnot \lnot 1=1,$ $\lnot \lnot 0=0$.
\item $0\leq \varepsilon \leq \varepsilon ^{\prime }\leq \delta \leq \delta^{\prime }\leq 1,$ $\lnot \varepsilon =\lnot \varepsilon ^{\prime }=\delta^{\prime },$ $\lnot \delta =\lnot \delta ^{\prime }=\varepsilon ^{\prime },$ thus $E=\{0,\varepsilon ,\varepsilon ^{\prime },\delta ,\delta ^{\prime},1\} $ is a subalgebra of $\alg{A}$.
\end{enumerate}

\noindent Expand $\alg{A}$ to the algebra $\alg{A}^{+}=\tuple{A,\wedge ,\vee ,\lnot,0,\varepsilon ,\varepsilon ^{\prime},\delta ,\delta^{\prime},1}$, and use the same symbols in the signature to denote the constants in $E$. By distributivity, (1), (5), (6) and (7), any term $t(v_{1},\ldots,v_{k})$ in the signature of $\alg{A}^{+}$ may be put in disjunctive normal form: $\bigvee\nolimits_{s}C_{s}$ with $C_{s}=e\wedge \bigwedge\nolimits_{r}\lnot^{n_{r}}v_{i_{r}},$ where $e\in E,$ $v_{i_{r}}\in \{v_{1},\ldots,v_{k}\}$ and $n_{r}\leq 2$. Moreover, by (2), we may assume no pair $\{v_{i},\lnot \lnot v_{i}\}$ appears in $C_{s}$, since $v_{i}\wedge \lnot \lnot v_{i}=v_{i}$. By idempotency and commutativity of $\wedge $, we may assume that each variable appears at most once in $C_{s}$ as: $v_{i},$ $\lnot v_{i},$ or $\lnot \lnot v_{i},$ or at most twice as: $v_{i}\wedge \lnot v_{i},$ or $\lnot v_{i}\wedge \lnot \lnot v_{i}$. Similarly, $t(v_{1},\ldots,v_{k})$ may be put in conjunctive normal form $\bigwedge\nolimits_{s}D_{s}$ with $D_{s}=e\vee\bigvee\nolimits_{r}\lnot ^{n_{r}}v_{i_{r}}$ and no occurrence of both $v_{i},\lnot \lnot v_{i}$. We prove now a result on elimination of quantifiers. Let $Ext(\ell )$ denote the conjunction of the extension axioms $Ext_\alg{A}(k,f)$ for $k\leq \ell ,$ and define for formulas $\varphi ,\psi \in \mathcal{L}_{\tau }^{\alg{A}^{+}}$: $\varphi \equiv _{Ext(\ell )}\psi$ iff $||\varphi(m_{1},\ldots,m_{k})||^{\mathfrak{M}}=||\psi (m_{1},\ldots,m_{k})||^{\mathfrak{M}}$
for any model $\mathfrak{M}$ of $\mathcal{L}_{\tau }^{\alg{A}^{+}}$ satisfying $Ext(\ell)$, and any $m_{1},\ldots,m_{k}\in M$.

\noindent \textbf{Claim}. Any formula $\varphi \in (\mathcal{L}_{\tau}^{\alg{A}^{+}})^{(\ell )}$ is $Ext(\ell )$-equivalent to a quantifier-free formula in the same free variables as $\varphi$.

\noindent We proceed by induction on the complexity of $\varphi$. Assume by induction hypothesis that $\varphi (x_{1},\ldots,x_{k},x_{k+1})\equiv _{Ext(\ell)}\psi (x_{1},\ldots,x_{k},x_{k+1})$ quantifier-free, and $k+1\leq \ell ,$ we must show the same for $\exists x_{k+1}\,\varphi (x_{1},\ldots,x_{k},x_{k+1})$ and $\forall x_{k+1}\,\varphi (x_{1},\ldots,x_{k},x_{k+1})$. By the previous algebraic remarks, we may assume $\psi $ is in disjunctive normal form, i.e.\\ $\bigvee\nolimits_{s}C_{s}(x_{1},\ldots,x_{k},x_{k+1})$ where each $C_{s}$ is a conjunction: 
\begin{equation*}
e\wedge C_{s}^{\prime }(x_{1},\ldots,x_{k})\wedge \bigwedge\nolimits_{r}\lnot^{n_{r}}R_{r}(x_{1},\ldots,x_{k},x_{k+1}),\text{ \ }n_{r}\in \{0,1,2\},
\end{equation*}
$C_{s}^{\prime }(x_{1},\ldots,x_{k})$ being the part not containing the variable $x_{k+1}$, such that each atomic $R_{r}(x_{1},\ldots,x_{k},x_{k+1})$ appearing in $C_{s}$ does it in one and only one of the following forms:

$R_{r}$ alone, $\lnot R_{r}$ alone, $\lnot \lnot R_{r}$ alone, $R_{r}\wedge \lnot R_{r},$ $\lnot R_{r}\wedge \lnot \lnot R_{r}$.

\noindent Substitute the first three types of formulas with (the symbol) $1$, the fourth with $\varepsilon,$ and the fifth with $\varepsilon^{\prime}; $ and let $c$ be the minimum of $e$ and the substituted constants to obtain:

$C_{s}^{+}(x_{1},\ldots,x_{k}):=c\wedge C_{s}^{\prime }(x_{1},\ldots,x_{k})$.

\noindent Since we have substituted the highest possible values in each case, we have for any model $\mathfrak{M}$ and $x_{1},\ldots,x_{k},i\in M$:

$||C_{s}(x_{1},\ldots,x_{k},i)||^{\mathfrak{M}}\leq ||C_{s}^{+}(x_{1},\ldots,x_{k})||^{\mathfrak{M}}$.

\noindent On the other hand, if $\mathfrak{M}$ satisfies $Ext(\ell )$ and $k+1\leq \ell ,$ there is, for any $x_{1},\ldots,x_{k}$ in $M$, an $i_{0}\in M$ such that for each $R_{r}$ appearing in $C_{s}(x_{1},\ldots,x_{k},x_{k+1}),$ the atomic formula $R_{r}(x_{1},\ldots,x_{k},i_{0})$ takes a value making the configuration in which it appears to take the highest possible value; that is,

$||C_{s}(x_{1},\ldots,x_{k},i_{0})||^{\mathfrak{M}}=||C_{s}^{+}(x_{1},\ldots,x_{k})||^{\mathfrak{M}}$.

\noindent Therefore, for any $x_{1},\ldots,x_{k}\in M$

$||\exists x_{k+1}\,C_{s}(x_{1},\ldots,x_{k},x_{k+1})||^{\mathfrak{M}}=\sup_{i\in M}||C_{s}(x_{1},\ldots,x_{k},i)||^{\mathfrak{M}}=\\=||C_{s}^{+}(x_{1},\ldots,x_{k})||^{\mathfrak{M}}$

\noindent and, finally,

$||\exists x_{k+1}\,\varphi (x_{1},\ldots,x_{k},x_{k+1})||^{\mathfrak{M}}=||\bigvee\nolimits_{s}\exists x_{k+1}\,C_{s}(x_{1},\ldots,x_{k},i)||^{\mathfrak{M}}=\\=||\bigvee\nolimits_{s}C_{s}^{+}(x_{1},\ldots,x_{k})||^{\mathfrak{M}},$

\noindent showing that $\exists x_{k+1}\,\varphi (x_{1},\ldots,x_{k},x_{k+1})\equiv _{Ext(\ell)}\bigvee\nolimits_{s}C_{s}^{+}(x_{1},\ldots,x_{k})$.

Symmetrically, one eliminates the quantifier in $\forall x_{k+1}\,\varphi (x_{1},\ldots,x_{k},x_{k+1})$ using a conjunctive normal form of $\psi $ and substituting in the elementary disjunction each configuration by its lowest possible value:

$R_{r}$ alone, $\lnot R_{r}$ alone, $\lnot \lnot R_{r}$ alone go to $0$,

$R_{r}\vee \lnot R_{r}$ goes to $\delta$,

$\lnot R_{r}\vee \lnot \lnot R_{r}$ goes to $\delta ^{\prime }$.

\noindent Finally, any sentence $\varphi \in (\mathcal{L}_{\tau}^{\alg{A}^{+}})^{(\ell )}$ is $Ext(\ell )$-equivalent to a quantifier-free sentence $\psi $ in $\mathcal{L}_{\tau }^{\alg{A}^{+}},$ necessarily a $\wedge,\vee ,\lnot $ combination of elements of $E$ and thus an element $e$ of $E$. Hence, $\lim_n \mu_{n}(\varphi =e)\geq \lim_n \mu_{n}(Ext(\ell))= 1$.
\end{proof}

Clearly, the conditions of the theorem are satisfied by De\ Morgan algebras, which include Boolean algebras and the $\{\wedge ,\vee ,\lnot\}$-reducts of \MV-algebras. They are satisfied also by the $\{\wedge,\vee,\lnot\}$-reduct of any G-algebra since the law $\lnot (v\wedge w)=\lnot v\vee \lnot w$ fails in arbitrary Heyting algebras but holds in prelinear ones. 
%[[[Hay que verificarlo: vale en cadenas y toda \'{a}lgebra de G\'{o}del es un subproducto de cadenas
%TCIMACRO{\U{a8}}%
%BeginExpansion
%
%EndExpansion

\begin{Exm}\label{ex:FiniteBoolean}
\emph{The almost sure values for any finite Boolean algebra are just $0$ and $1$,
since $\varepsilon =\varepsilon ^{\prime }=0$ and $\delta =\delta ^{\prime}=1$.}
\end{Exm}

\begin{Exm}\label{Lukas}
\emph{The almost sure values for the $\{\wedge,\vee,\lnot\}$-fragment of $\mathcal{L}^{\mathrmL _{N+1}}$ are $\{0,\frac{1}{2},1\}$ for even $N$, and $\{0,\frac{N-1}{2N},\frac{N+1}{2N},1\}$ for odd $N$. Indeed, $\varepsilon =\varepsilon ^{\prime}=\delta =\delta^{\prime }=\frac{N/2}{N}$ in the first case, and $\varepsilon =\varepsilon ^{\prime }=\max(v\wedge (1-v))=\frac{(N-1)/2}{N}$, $\delta =\delta ^{\prime }=\min (v\vee(1-v))=\frac{(N+1)/2}{N}$ in the second.}
\end{Exm}

\begin{Exm}
\emph{For any finite \G-algebra $\alg{G}$, the almost sure values of the $\{\wedge,\vee,\lnot\}$-fragment of $\mathcal{L}^{\alg{G}}$ are $0,\delta,1,$ where $\delta =\inf (v\vee \lnot v)$, since $\varepsilon =\varepsilon^{\prime}=0$ and $\delta^{\prime}=\inf (\lnot v\vee \lnot \lnot v) = 0$. Notice that $\delta$ is the minimum dense element of $\alg{G}$ and the smallest positive element if $\alg{G}$ is a chain.}
\end{Exm}

\begin{Exm}
\emph{The almost sure values for a finite De Morgan algebra are at most four since $\varepsilon =\varepsilon ^{\prime }$ and\ $\delta =\delta ^{\prime }$. Examples~\ref{ex:FiniteBoolean} and \ref{Lukas} show they may be exactly two, exactly three or exactly four.}
\end{Exm}

\begin{Exm}
\emph{As the conditions of the theorem are equational, they are preserved by finite products; in particular, the reader may verify that the logic associated to the algebra $\G_{3}\times \L_{4}$ has 5 distinct almost sure values: $0,\varepsilon=\varepsilon^{\prime},\delta,\delta^{\prime},1$. We let the reader find an example where six distinct values are achieved.}
\end{Exm}

Notice that no new values are added by $(\mathcal{L}^\alg{A})_{\omega _{1}\omega }^{\omega}$ because under the extension axioms any sentence of this logic is equivalent to a finitary one (see e.g.~\cite[Proposition 3.3.2]{flum}).

\section{A zero-one law for infinitely valued \L ukasiewicz predicate logic}\label{s:01-law-Luk}

Proving zero-one laws for logics given by infinite lattice algebras seems challenging. The translation yields in this case infinitary formulas with essentially infinite vocabularies an it is known that in the case of vocabularies with an infinite supply of unary predicate symbols, the zero-one law for classical  $\mathcal{L}_{\omega_{1}\omega }^{\omega}$ fails~\cite[Exercise 4.1.9.]{flum}. Moreover, we cannot expect positive probabilities $p_{R}(a)$ for all the truth-values (except  for countable $A$).

We will deal mainly with infinitely valued \L ukasiewicz logic, determined by the standard \MV-algebra $[0,1]_\mathrmL=\tuple{[0,1],\wedge ,\vee,\rightarrow,\lnot}$.

Consider $\sigma$-additive Borel probability measures $\{p_{R}\}_{R\in \tau}$ over $[0,1]$ which assign positive measure to all non-trivial intervals\ of $[0,1]$, say, $p_{R}=$ Lebesgue measure.
These induce a product measure $\rho _{n}$ in the infinite space of models with domain $[n]$: $\Mod_{\tau ,n}^{[0,1]}=\prod_{R\in \tau}[0,1]^{[n]^{\tau(R)}}\!$, which permits to define for any measurable set $S \subseteq [0,1]$ the probability that a sentence $\varphi $ takes a value in $S:$
\begin{equation*}
\mu_{n}(\varphi \in S)=\rho _{n}\{\mathfrak{M}\in \Mod_{\tau,n}^{[0,1]}:\varphi ^{\mathfrak{M}}\in S\}
\end{equation*}
(the map $\mathfrak{M}\longmapsto \varphi^{\mathfrak{M}}$ is measurable because it may be shown to be continuous).

It is important to stress that for each predicates $R$ and $R^{\prime}$, and $i_{1},\ldots,i_{k},i_{1}^{\prime},\ldots,i_{k}^{\prime}\in \lbrack n]$, we have that $R(i_{1},\ldots,i_{k})$ and $R^{\prime}(i_{1}^{\prime},\ldots,i_{k}^{\prime})$ become independent events provided that $\tuple{R,i_{1},\ldots,i_{k}}$ and $\tuple{R^{\prime},i_{1}^{\prime },\ldots,i_{k}^{\prime}}$ are distinct. 

For simplicity, we will use the notation $\mathfrak{M}\models \varphi$ as an equivalent of $||\varphi||^\mathfrak{M} = 1$.

%Recall that the intervals $[0,u),$ $(u,1]$ are open in $[0,1].$

\begin{Thm}\label{zeroonelaw1}
For any sentence $\varphi \in \mathcal{L }^{[0,1]_\mathrmL}$, there is a unique $a\in \lbrack
0,1]$ such that
\begin{equation*}
\lim_{n}\mu _{n}(\varphi \in V)=1 \ \ \text{for \ any \ open interval} \ V\text{\ containing\ }a.
\end{equation*}
Equivalently, \begin{equation*}
\lim_{n}\mu _{n}(|\varphi -a|<\varepsilon )=1\text{ \ \ }\emph{\ for\ any}
\text{ }\varepsilon >0.
\end{equation*}
\end{Thm}

The two formulations are clearly equivalent. They say that $\varphi $ takes approximately value $a$ with probability $1$, for any degree of approximation.

To prove the theorem, we will need the following observations on the expressivity of the logic (see \cite{bell, chang, mun}). For each fraction $r=\frac{n}{m} \in [0,1]$, there are connectives of \L ukasiewicz logic $($ $)_{\geq r},$ $($ $)_{\leq r}$ such that 

$\model{M} \models  \varphi _{\geq r}$ iff $||\varphi||^{\model{M}}\geq r$,

$\model{M} \models  \varphi _{\geq r}$ iff $||\varphi||^{\model{M}}\leq r$.

\noindent More generally and precisely (see~\cite{cai}), 

$||\varphi_{\geq r}||^{\model{M}}\geq 1-\varepsilon$ iff $||\varphi||^{\model{M}}\geq r-\frac{\varepsilon }{m}$

$||\varphi_{\leq r}||^{\model{M}}\geq 1-\varepsilon$ iff $||\varphi||^{\model{M}}\leq r+\frac{\varepsilon }{m}$ 

\noindent 
We will assume also that the logic $\mathcal{L}^{[0,1]_\mathrmL}$ has a `crisp identity' $\approx $. \footnote{This does not imply loss of generality because any model of $\L$ may be equipped with a crisp identity, without altering the model theoretic properties of the logic.}

Given $k,N\in \omega ,$ let $F_{k+1}$ consist of all the atomic formulas distinct from identities truly containing the variable $x_{k+1},$ and $A_{N}=\{[\frac{j}{N},\frac{j+1}{N}],$ $j=0,\ldots, N-1\}$. For each choice of intervals $g \colon F_{k+1}\rightarrow A_{N}$, define the extension axiom $Ext_{A}(k,N,g)$:

$\forall x_{1}\ldots x_{k}\,(\bigvee _{i\not=j}x_{i}\approx x_{j}\vee \exists x_{k+1}\,\bigwedge _{i\leq k}\lnot x_{k+1}\approx x_{j}\wedge { \bigwedge }_{\varphi \in F_{k+1}}\varphi _{g(\varphi)}(x_{1},\ldots, x_{k},x_{k+1}))$

\noindent Then define the sentence $Ext_{A}(k,N)$:

$\bigwedge_{g \colon F_{k+1}\rightarrow A_{N}}Ext_{A}(k,N,g)$

\noindent and the theory

$Ext:=\{Ext_{A}(k,N)\}_{k,N}$.

Due to the semantical interpretation of quantifiers as infima and suprema, we have that $Ext_{A}(k,N,g)^{M}\geq
1-\varepsilon $ if and only if for any distinct $i_{1},\ldots ,i_{k}$ in $M$
there is an $i$ in $M\setminus \{i_{1},\ldots,i_{k}\}$ such that $||\varphi
_{g(\varphi )}||^{\mathfrak{M}}(i_{1},\ldots ,i_{k},i)\geq 1-\varepsilon $ for each $\varphi \in F_{k+1}$. More precisely, if $g(\varphi )=[\frac{j}{N},\frac{j+1}{N}],$ 
\begin{equation*}
||\varphi ||^\mathfrak{M}(i_{1},\ldots ,i_{k},i)\in \lbrack \frac{j}{N}-\frac{\varepsilon }{m},\frac{j+1}{N}+\frac{\varepsilon }{m}],
\end{equation*}
for each $\varphi \in F_{k+1}$, which is a relaxation of the intended condition $$||\varphi ||^\mathfrak{M}(i_{1},\ldots ,i_{k},i_{k+1})\in \lbrack \frac{j}{N},\frac{j+1}{N}].$$ In the following, we will denote by $g(\varphi )_{\pm \varepsilon}$ the displayed relaxed interval.

\begin{Lem}
For any $\varepsilon >0$, the sequence $\mu_{n}(Ext_{A}(k,N,g)\geq 1-\varepsilon)$ converges to $1$.
\end{Lem}

\begin{proof}
Fix $k,N,$ $g \colon F_{k+1}\rightarrow A_{N}$ and notice
that according to the previous observation $Ext_{A}(k,g,N)^{M}<1-\varepsilon 
$ means that there are distinct $i_{1},\ldots ,i_{k}$ in $[n]$ such
that:

for all $i\in \lbrack n]\setminus \{i_{1},\ldots,i_{k}\}$, it is false
that $\bigwedge\nolimits_{\varphi \in F_{k+1}}||\varphi ||^{\mathfrak{M}}(i_{1},\ldots
,i_{k},i)\in g(\varphi)_{\pm \varepsilon}$.

\noindent Let $S(M,\varepsilon ,i_{1},\ldots ,i_{k})$ be the above classical
statement for fixed $\varepsilon $ and $i_{1},\ldots ,i_{k}$, and $\mathfrak{M}$ chosen
randomly in $\Mod_{\tau,n}^{[0,1]}$. Then, since the set of events

\begin{equation*}
\{||\varphi ||^{\mathfrak{M}}(i_{1},\ldots ,i_{k},i)\in g(\varphi )_{\pm \varepsilon
}:\varphi \in F_{k+1},\text{ }i\in \lbrack n]\setminus \{i_{1},\ldots
,i_{k}\}\}
\end{equation*}
is independent,

$\mu_n(S(M,\varepsilon ,i_{1},\ldots ,i_{k}))=\prod\nolimits_{i\in
\lbrack n]\setminus \{i_{1},\ldots ,i_{k}\}}(1-\prod\nolimits_{\varphi
\in F_{k+1}}\mu_n (\varphi(i_{1},\ldots ,i_{k},i)\in g(\varphi )_{\pm
\varepsilon }))$ 

\noindent But $\mu_n (\varphi(i_{1},\ldots ,i_{k},i)\in g(\varphi)_{\pm \varepsilon })=p_{R}(g(\varphi )_{\pm \varepsilon })\geq p_{R}([\frac{j}{N},\frac{j+1}{N}])>0,$ for some relation symbol $R$ and some $j$. Then, taking $\delta _{N}=\min \{p_{R}([\frac{j}{N},\frac{j+1}{N}]):R\in \tau ,$ $0\leq j\leq N-1\},$ we have 
\begin{equation*}
\mu _{n}(S(M,\varepsilon ,i_{1},\ldots ,i_{k}))\leq (1-\delta _{N})^{n-k},
\end{equation*}
Therefore,

$\mu _{n}(Ext_{A}(k,g,N)<1-\varepsilon )\leq \sum_{i_{1},\ldots
,i_{k,}}\mu _{n}(S(M,\varepsilon ,i_{1},\ldots ,i_{k}))\leq n^{k}(1-\delta
_{N})^{n-k}$

\noindent which clearly converges to $0,$ since $1-\delta _{N}<1$.
Therefore, $\mu _{n}(Ext_{A}(k,g,N)\geq 1-\varepsilon )$ converges to $1$.
\end{proof}

Recall that $Ext_{A}(k,N,g)^{\mathfrak{M}}\geq 1-\varepsilon $ iff $M\models Ext_{A}(k,N,g)_{\geq 1-\varepsilon }$. Therefore, the previous lemma shows that any finite subset of the theory $Ext^{\ast }=\{\psi _{\geq 1-\varepsilon }:\psi \in Ext,$ $\varepsilon >0)$ is satisfiable. Hence, by satisfiability compactness of \L ukasiewicz logic it has a model, which will be necessarily a model of $Ext$. As $Ext$ is countable, this model may be chosen to be countable by the downward L\"{o}wenheim--Skolem theorem.

\noindent \begin{Def}\emph{Let \emph{\ }$\overline{a},\overline{b}$ be tuples of the same length $k$ of distinct elements of $\mathfrak{A}$ and $\mathfrak{B},$ respectively, then $\overline{a}\sim _{\varepsilon }\overline{b}$ if and only if $|\varphi ^{\mathfrak{A}}(\overline{a})-\varphi ^{\mathfrak{B}}(\overline{b})|\leq \varepsilon $ for any atomic $\varphi $ with variables in $\{x_{1},\ldots,x_{k}\}$. We say then that the pair $\overline{a},\overline{b}$ (or the assignment $a_{i}\mapsto b_{i})$ is an $\varepsilon $-\emph{partial isomorphism}.}
\end{Def} 

Considering tuples enumerating $\mathfrak{A}$ and $\mathfrak{B}$ respectively, we may define, similarly, the notion of $\varepsilon$-isomorphism. Partial $\varepsilon$-isomorphisms between models of $Ext$ have the extension property, more precisely.

\begin{Lem}\label{extension} If\/ $\mathfrak{A},\mathfrak{B}\models Ext(k,N)$, $N\geq \frac{2}{\varepsilon}$,  then $\overline{a}\sim _{\varepsilon }\overline{b}$ (of length $\leq k)$ implies for any $a\in A\setminus \{\overline{a}\}$
there is $b\in B\setminus \{\overline{b}\}$ such that $\overline{a}a\sim _{\varepsilon }\overline{bb}$ and viceversa. 
Particular case: for any $a\in A$, there is $b\in B$ such that $a\sim_{\varepsilon} b$.
\end{Lem} 

\begin{proof}
Assume $\overline{a}\sim _{\varepsilon }\overline{b},$ and consider $g \colon F_{k+1}\rightarrow A_{N}$ such that $\varphi ^{\mathfrak{A}}(\overline{a},a)\in g(\varphi )$ for each $\varphi \in F_{k+1}$, then $Ext(k,N,g)$ is an extension axiom which must be satisfied by $\mathfrak{B};$ hence, there is $b$ such that $||\varphi _{g(\varphi )}||^{\mathfrak{B}}(\overline{b},b)\geq 1-\varepsilon /2$. This means that $\varphi ^{\mathfrak{B}}(\overline{b},b)\in \lbrack \frac{j}{N}-\varepsilon/2,\frac{j+1}{N}+\varepsilon/2],$ thus $|\varphi^{\mathfrak{A}}(\overline{a},a)-\varphi ^{\mathfrak{B}}(\overline{b},b)|\leq \frac{1}{N}+\frac{\varepsilon}{2}\leq \varepsilon|$.
\end{proof}

We will utilize the following observation:\ $|r\oplus s-r^{\prime}\oplus s^{\prime}|\leq 2\max \{|r-r^{\prime }|,|s-s^{\prime }|\}$. For any formula $\varphi ,$ let $\#(\varphi )$ be the number of occurrences of $\oplus \ $in\ $\varphi $. 

\begin{Lem}\label{continu} If\/ $\mathfrak{A},\mathfrak{B}\models Ext(k,N)$\emph{, }$N\geq \frac{2}{\varepsilon },$ then for any formula $\varphi $ with free variables in $\{x_{1},\ldots,x_{k}\}$, \emph{\ }$\overline{a}\sim _{\varepsilon }\overline{b}$  implies $|\varphi^{\mathfrak{A}}(\overline{a})-\varphi ^{\mathfrak{B}}(\overline{b})|\leq 2^{\#(\varphi )}\varepsilon$.\end{Lem} 

\begin{proof}
By induction on the complexity of $\varphi (\overline{x})\in \L\forall (\tau),$ based in $\oplus ,$ $\lnot ,$ $\exists$.

\begin{itemize}

\item If $\varphi$ is atomic, then the result is immediate by definition of partial $\varepsilon$-isomorphism.

\item If $\varphi $ is $\lnot \theta $, then it is trivial because $|(\lnot \theta )^{\mathfrak{A}}[\overline{a}]-(\lnot \theta )^{\mathfrak{B}}[\overline{b}]|=|\theta ^{\mathfrak{A}}[\overline{a}]-\theta ^{\mathfrak{B}}[\overline{b}]|.$

\item If $\varphi $ is $\theta \oplus \theta ^{\prime }$, then, by induction hypothesis, $|\theta ^{\mathfrak{A}}[\overline{a}]-\theta ^{\mathfrak{B}}[\overline{b}]|\leq 2^{\#(\theta )}\varepsilon ,$ $|\theta ^{\prime \mathfrak{A}}[\overline{a}]-\theta ^{\prime \mathfrak{B}}[\overline{b}]|\leq 2^{\#(\theta ^{\prime })}\varepsilon $ and

$|(\theta \oplus \theta ^{\prime })^{\mathfrak{A}}[\overline{a}]-(\theta \oplus \theta ^{\prime })^{\mathfrak{B}}[\overline{b}]|\leq 2\cdot 2^{\max(\#(\theta ),\#(\theta ^{\prime }))}\varepsilon \leq 2^{\#(\theta)+\#(\theta ^{\prime })+1}\varepsilon =2^{\#(\varphi )}\varepsilon .$

\item If $\varphi $ is $\exists v\,\theta$ and $\overline{a}\sim _{\varepsilon }\overline{b}$, then, by the previous lemma, for each $a\in A$ there is $b_{a}\in B$ such that $\overline{a}a\sim _{\varepsilon }\overline{b}b_{a},$ and the inductive hypothesis yields $|\theta ^{\mathfrak{A}}[\overline{a}a]-\theta ^{\mathfrak{B}}[\overline{b}b_{a}]|\leq 2^{\#(\theta )}\varepsilon.$ Hence, $\theta ^{\mathfrak{A}}[\overline{a}a]\leq \theta^{\mathfrak{B}}[\overline{b}b_{a}]+2^{\#(\theta )}\varepsilon$. Taking suprema over $a$, we obtain $(\exists v\,\theta )^{\mathfrak{A}}[\overline{a}]\leq\sup_{b_{a}}\theta^{\mathfrak{B}}[\overline{b}b_{a}]+2^{\#(\theta)}\varepsilon \leq (\exists v\,\theta)^\mathfrak{B}+2^{\#(\theta)}\varepsilon [\overline{b}],$ and thus $(\exists v\,\theta )^{\mathfrak{A}}[\overline{a}]-(\exists v\,\theta )^\mathfrak{B}[\overline{b}]\leq 2^{\#(\theta )}\varepsilon .$ Symmetrically, we get the same for the negation and thus $|(\exists v\,\theta)^{\mathfrak{A}}[\overline{b}]-(\exists v\theta )^{B}[\overline{a}]|\leq 2^{\#(\theta)}\varepsilon =2^{\#(\varphi )}\varepsilon$.\qedhere
\end{itemize}
\end{proof}  

We will write $\mathfrak{A}\equiv_{\mathcal{L }^{[0,1]_\mathrmL}}\mathfrak{B}$ whenever $\varphi ^{\mathfrak{A}}=\varphi ^{\mathfrak{B}}$ for each sentence $\varphi$.

\begin{Cor}[Completeness] If\/ $\mathfrak{A},\mathfrak{B}\models Ext$, then $\mathfrak{A}\equiv _{\mathcal{L}^{[0,1]_\mathrmL}}\mathfrak{B}$. Therefore, $Ext$ is complete.
\end{Cor} 

\begin{proof} Given a sentence $\varphi $ and $\varepsilon >0,$ let $x_{1},\ldots,x_{k}$ be its variables and $N\geq 2\cdot 2^{\#(\varphi)}/\varepsilon$. Then, by Lemma~\ref{extension}, there are $a\in A,$ $b\in B$ such that $a\sim _{\varepsilon /2^{\#(\varphi )}}b$ and, by Lemma~\ref{continu}, $|\varphi ^{\mathfrak{A}}-\varphi ^{\mathfrak{B}}|=|\varphi ^{\mathfrak{A}}[a]-\varphi ^{\mathfrak{B}}[b]|\leq 2^{\#(\varphi )}\varepsilon /2^{\#(\varphi )}=\varepsilon$. Since this is true for all $\varepsilon >0$, we have $\varphi ^{\mathfrak{A}}=\varphi ^{\mathfrak{B}}$.
\end{proof}

It is easy to see that a theory $T$ is complete iff for any sentence $\varphi $ and rational $r$, we have $T\models \varphi _{\geq r}$ or $T\models \varphi _{\leq r}.$ 

\begin{proof}(of Theorem~\ref{zeroonelaw1}) Fix a sentence $\varphi$. By completeness, one has $Ext\models \varphi _{\geq r}$ or $Ext\models \varphi_{\leq r}\,$ for each rational $r\in \lbrack 0,1]$. Then, the sets $L=\{r\in \mathbb{Q}:Ext\models \varphi _{\geq r}\},$ $U=\{s\in \mathbb{Q}:Ext\models \varphi _{\leq s}\}$ are easily seen to determine a real cut $a\in\lbrack 0,1]$. Assume first $a\in (0,1)$ and let $a\in (r,s)$ and  $r<r^{\prime }<a<s^{\prime }<s$. Then, $Ext\models \varphi _{\geq r^{\prime}}\wedge \varphi _{\leq s^{\prime }}$ by construction. Take a positive $\varepsilon <s-s^{\prime},rs^{\prime}-r$. By approximate consequence compactness of \L, there is a subset $Ext^{\ast}\subseteq _{fin}Ext$ such that $Ext^{\ast }\models (\varphi _{\geq r^{\prime }}\wedge \varphi _{\leq s^{\prime }})_{\geq1-\varepsilon }$ , thus $Ext^{\ast }\models \varphi _{\lbrack r^{\prime }-\varepsilon ,\text{ }s^{\prime }+\varepsilon ]}$. Since $\lim_n \mu _{n}(Ext^{\ast})= 1$, we have $\lim_n \mu _{n}(\varphi \in \lbrack r^{\prime }-\varepsilon,s^{\prime }+\varepsilon ])= 1$ and, a fortiori, $\lim_n\mu_{n}(\varphi \in (r,s))= 1$. This shows the existence of $a$. Uniqueness follows easily since $b\not=a$ implies the existence of rationals $u,v,$\ $r$ such that $u<b<r<a<s.$ As $\varphi \in (r,s)$ has asymptotic probability $1$, the incompatible event $\varphi \in (u,r)$ must have asymptotic probability $0$.

If $a\in \lbrack 0,s)$, choose $a<s^{\prime }<s$. Then, $Ext\models \varphi _{\leq s^{\prime}}$ and $Ext^{\ast }\models (\varphi _{\leq s^{\prime }})_{1-\varepsilon }$ for an $\varepsilon <s-s^{\prime }$, and hence $Ext^{\ast }\models \varphi _{\lbrack 0,s^{\prime }+\varepsilon ]}$. Since $\lim_n\mu_{n}(Ext^{\ast })= 1$, we have $\lim_n\mu _{n}(\varphi \in \lbrack 0,s^{\prime }+\varepsilon ])= 1$, and, a fortiori, $\lim_n \mu_{n}(\varphi \in \lbrack 0,s))= 1$. The case $a=1$ is similar.
\end{proof}

\begin{Rmk} \emph{The inductive step in Lemma~\ref{continu} works without extra effort for any continuous connective $\circ$ under the rank definition: $\#(\circ(\theta_{1},\ldots,$ $\theta _{k}))=\max_{i}\#(\theta_{i})+1$. So $Ext$ is complete for any expansion of \L ukasiewicz logic by continuous connectives, and thus for full continuous logic, and the same is true of the zero-one law.}\end{Rmk} 

\begin{Rmk} \emph{From lemmas~\ref{extension} and~\ref{continu}, one may show that, for any countable $\mathfrak{A},\mathfrak{B}\models Ext$ and $\varepsilon >0$, there are enumerations $a_{0},a_{1},\ldots$ of $A$ and $b_{0},b_{1},\ldots$ of $B$ such that $a_{i}\mapsto b_{i}$ is an $\varepsilon$-isomorphism. This enumerations depend on $\varepsilon$ and it is not possible to obtain $\omega $-categoricity in this way. However, it follows from the above proof  that any sentence $\phi$ takes in all models of $Ext$ the same value, precisely its almost sure truth-value. }\end{Rmk}

\section{Further results for infinitely valued predicate logics}\label{s:Further-infinitely}

In this section, although we cannot expect a complete description of the infinitely valued case, we still may obtain some results generalizing those from in Section~\ref{s:SetAlmostSure} for the finitely valued case.  That is, we consider again the problem of describing sets of almost sure values, now for infinitely valued logics. In particular, we obtain that all rational numbers in $[0,1]$ are almost sure values of infinitely valued \L ukasiewicz predicate logic and, after that, we prove a zero-one law for a large family of $[0,1]$-valued logics satisfying De Morgan laws.

Observe that the analogue of Lemma~\ref{al} holds for infinitely valued \L ukasiewicz logic with a similar proof, due to the continuity of McNaughton functions.

\begin{Thm}\label{t:AlmostSureLuk}
Let $t(v_{1},\ldots,v_{k})$ be a term in the signature of\/ \MV-algebras. Then, $m=\inf_{a_{1},\ldots,a_{k}\in A}t^\alg{A}(a_{1},\ldots,a_{k})$ and $s=\sup_{a_{1},\ldots,a_{k}\in A}t^\alg{A}(a_{1},\ldots,a_{k})$ are almost sure for $[0,1]_\mathrmL$.
\end{Thm}

\begin{proof}
Notice first that $m$ and $s$ exist by compactness of $[0,1]$ and the continuity of $t^\alg{A}$. Let  $R(x)$ be an atomic formula in one variable. We show that for any open interval $V$ of $[0,1]$ containing $m$:
\begin{equation*}
\mu _{n}(\forall x_{1}\ldots \forall x_{k}\,(t(R(x_{1}),\ldots,R(x_{k})))\in V)\ \text{converges to}\ 1.
\end{equation*}

By definition of $m$, there is a tuple $\tuple{a_{j}}_{j}\in |A|^{k}$ such that $t^\alg{A}(a_{1},\ldots,a_{k})\in V$ and  by continuity of $t^\alg{A}$, there are open intervals $V_{j}$ such that $a_{j}\in V_{j}$ and $v_{j}\in V_{j}$ implies $t^\alg{A}(v_{1},\ldots,v_{k})\in V.$

For each $n\geq k$ find $K\subseteq \lbrack n]^{k}$ of power $|K|=\lfloor n/k\rfloor $ such that all tuples in $K$ have all their components distinct, and two different tuples in $K$ are distinct at each coordinate. Then,

$\mu _{n}(\forall x_{1}\ldots\forall x_{k}\,(t(R(x_{1}),\ldots,R(x_{k}))\in V)$

$=\Pr (\inf_{i_{1},\ldots,i_{k}\in \lbrack n]}t^\alg{A}(R(i_{1}),\ldots,R(i_{k}))\in V)$

$=\Pr ($for some $(i_{j})_{j}$ in $[n]^{k}:t^\alg{A}(R(i_{1}),\ldots,R(i_{k}))\in V)$ (since $V$ is an open interval containing $m$)

$\geq \Pr ($for some $(i_{j})_{j}$ in $[n]^{k}:R(i_{j})\in V_{j},$ for $j=1,\ldots,k)$.

$\geq \Pr ($for some $(i_{j})_{j}$ in $K:R(i_{j})\in V_{j},$ for $j=1,\ldots,k)$

$=1-\Pr ($for all $(i_{j})_{j}$ in $K$ is false: $R(i_{j})\in V_{j},$ for $j=1,\ldots,k)$

$=1-\prod_{(i_{1}\ldots i_{k})\in K}[1-\Pr (R(i_{j})\in V_{j}$ for $j=1,\ldots,k)]$

$=1-(1-p_{R}(V_{1})\ldots p_{R}(V_{k}))^{\lfloor n/k\rfloor}$.

For the last two identities, we use that the events $R(i_{j})\in V_{j}$ form a mutually independent set for the coordinates of the $k$-tupes in $K$. Since $p_{R}(V_{i})>0$, the last quantity converges to $1$ when $n$ goes to $\infty$.

A similar argument shows that, for any open set $V$ containing $s$, $$\mu _{n}(\exists x_{1}\ldots \exists x_{k}\,(t(R(x_{1}),\ldots,R(x_{k}))\in V)$$ converges to $1$.
\end{proof}

\begin{Cor}
All rational numbers are almost sure values of $[0,1]$-valued \L ukasiewicz
logic ($\mathcal{L}^{[0,1]_\mathrmL}$).
\end{Cor}

\begin{proof}
Use the same terms of Example~\ref{Lukas}.
\end{proof}

It is easy to see that Theorem~\ref{t:AlmostSureLuk} generalizes to any bounded lattice chain for which the additional operations are continuous with respect to the order topology.

\begin{Thm}
Let $t(v_{1},\ldots,v_{k})$ be a term for a chain $\alg{A}$. If $t^\alg{A}$ is continuous w.r.t.\ the order topology, then $m=\inf_{a_{1},\ldots,a_{k}\in A}t^\alg{A}(a_{1},\ldots,a_{k})$ and $s=\sup_{a_{1},\ldots,a_{k}\in A}t^\alg{A}(a_{1},\ldots,a_{k})$ are almost sure for $\mathcal{L}^\alg{A}$ whenever they exist.
\end{Thm}

\begin{Exm}\emph{Consider \L ukasiewicz logic with product. Then, $\frac{1}{2}(\sqrt{5}-1)=\inf (\lnot v\vee (v\cdot v))$ is an almost-sure value. }\end{Exm}

\begin{Exm}\emph{Let $\alg{G}_\uparrow $ be the G-chain consisting of an ascending sequence from $0$ to $1$. Then, it may be checked that G\"{o}del implication is continuous in this case. Therefore, all elements of $\alg{G}_\uparrow$ are almost sure (compare with Example~\ref{ex:ReductG} below).}\end{Exm}

Finally, notice that the analogue of Lemma~\ref{De Morgan} holds for infinite compact sublattices of $[0,1]$ with a continuous negation, and actually gives a proof of the zero-one law in this case.

\begin{Thm}\label{De Morgan2}
Let $\alg{A}=\tuple{A,\wedge ,\vee ,\lnot}$ be an infinite compact sublattice of\/ $[0,1]$ containing $0$ and $1$ with a continuous negation satisfying:

\begin{enumerate}
\item $\lnot (v\wedge w)=\lnot v\vee \lnot w,$ $\lnot (v\vee w)=\lnot
v\wedge \lnot w$

\item $v\leq \lnot \lnot v$

\item $\lnot 1=0$
\end{enumerate}

Then, $\mathcal{L}^\alg{A}$ satisfies the zero-one law and its only almost sure values are $0$, $1$ and

$\varepsilon =\sup_{x\in A}(x\wedge \lnot x),$ $\varepsilon ^{\prime}=\sup_{x\in A}(\lnot x\wedge \lnot \lnot x)$,

$\delta =\inf_{x\in A}(x\vee \lnot x),$ $\delta ^{\prime}=\inf_{x\in A}(\lnot x\vee \lnot \lnot x)$.
\end{Thm}

\begin{proof}
By compactness and continuity of the operations, this infima and suprema are attained, for example $\varepsilon =(a\wedge \lnot a)$ for some $a \in A$. Thus, only finitely many extension axioms
are needed to ensure that the suprema of the various configurations $R$, $\lnot R$, $\lnot \lnot R$, $R\wedge \lnot R,$ $\lnot R\wedge \lnot \lnot R$ appearing in the proof of Lemma~\ref{De Morgan} are realized, and the quantifier elimination may be achieved with asymptotic probability $1$.
\end{proof}

\begin{Exm} \emph{The logic given by the $\{\land, \lor, \neg\}$-reduct of the \L ukasiewicz algebra $([0,1]\cap \mathbb{Q})_\mathrmL$ has almost sure values $\{0,\frac{1}{2},1\}.$}
\end{Exm}

\begin{Exm}\label{ex:ReductG} \emph{The logic given by the $\{\wedge, \vee, \neg\}$-reduct of any closed subalgebra of the standard \G-chain over $[0,1]$ in which $0$ is
isolated has a zero-one law and its almost sure values are $\{0,\delta ,1\}$ where $\delta$ is the minimum positive element since $\wedge, \vee$ are continuous and under isolation of $0$. G\"{o}del negation becomes continuous.}
\end{Exm}

\section{Final remarks}\label{s:Conclusion}

The translation we have used in the case of finite algebras has appeared in~\cite{bad, bad2} and permits to transfer straightforwardly from classical to finitely valued logics most relevant model-theoretic facts and concepts: compactness, ultraproducts, elementary embeddings, type omission, etc.

Gr\"{a}del et al~\cite{gra} consider certain semirings $\alg{A}=\tuple{A,+,\cdot ,0,1}$ as algebras of truth values in which  $\vee $ and $\wedge$ are interpreted as $+$ and $\cdot$, and the quantifiers $\exists$ and $\forall $ are interpreted as iterated $+$ and $\cdot$, respectively. Negation is allowed at the atomic level only, and no further operations (connectives) are allowed. In this setting they obtain zero-one laws for finite distributive lattices and ``absorptive'' semirings such as $\tuple{\L
_{M+1},\vee ,\odot ,0,1}$. Our results in the present paper generalize those on lattice semirings because we do not require distributivity and we have arbitrary additional operations. Moreover, the treatment of negation in~\cite{gra} can be expressed in terms of Oberschelp conditions.

Notice that the zero-one law fails for the simplest non-lattice semirings, namely the two-element field $\mathbb{F}_{2}=\tuple{\{0,1\},+,\cdot,0,1}$, as witnessed by the sentence $\exists x(Px\vee \lnot Px)$ whose value in
universes of power $n$ is $\Sigma_{x\in \lbrack n]}1$, which
oscillates between $0$ and $1$ with the parity of $n$.

It is worth noting that, in the context of continuous model theory, similar laws to the one we have provided here for infinitely valued \L ukasiewicz predicate logic have recently attracted attention in~\cite{gold} under certain restrictions for identity. In fact, in the present paper we have generalized these results to vocabularies containing other predicates than just identity.

We conclude with some questions that merit further investigation:

\begin{enumerate}
\item Does $\mathcal{L}^{[0,1]_\mathrmL}$ have almost sure irrational values?
\item Does full \L ukasiewicz logic over $[0,1]$ with product and its (discontinuous) residuated implication have a zero-one law?
\item Does full G\"{o}del logic over $[0,1]$ have a zero-one law?
\item The set $A_{as}$ of almost sure values of a logic $\mathcal{L}^\alg{A}$ is an invariant of $\alg{A}$, in fact, a subalgebra. Is there a characterization of $A_{as}$ in purely algebraic terms?
\end{enumerate}

\section*{Acknowledgments}
Badia was supported by the Australian Research Council grant DE220100544. Badia and Noguera were also  supported by the European Union's Marie Sklodowska--Curie grant no.\ 101007627 (MOSAIC project).

\end{document}